\documentclass[12pt]{amsart}  %para ARXIV

\usepackage{graphicx}
\usepackage{mathptmx}      % use Times fonts if available on your TeX system
\usepackage{amssymb,amsmath}
\usepackage[cmtip,all]{xy}
\usepackage{tikz-cd}
\usepackage[english,activeacute]{babel}

\usepackage{color}
\usepackage{soul} 
\setstcolor{blue}
\newtheorem{theorem}{Theorem}[section]    % Standard theorem environment
\newtheorem{proposition}[theorem]{Proposition} 
\newtheorem{corollary}[theorem]{Corollary}  

\theoremstyle{definition}
\newtheorem{definition}[theorem]{Definition}    % Definition environment with 
\newtheorem*{remark}{Remark}             % Unnumbered environment for remarks.
\newtheorem{example}[theorem]{Example}

\newcommand{\C}{\mathbb{C}}
\newcommand{\Class}{\mathcal{C}}
\newcommand{\cat}{\mathop{\mathrm{cat}}}
\newcommand{\crit}{\mathop{\mathrm{Crit}}}
\newcommand{\Cut}{\mathop{\mathrm{Cut}}}
\newcommand{\Di}{{\mathrm{D}}}

\newcommand{\Exp}{\mathrm{Exp}}

\newcommand{\grad}{\mathop{\mathrm{grad}}}
\newcommand{\Hess}{\mathop{\mathrm{Hess}}}

\newcommand{\id}{\mathrm{id}}
\newcommand{\NP}{\mathrm{NP}}
\newcommand{\Pa}{\mathrm{P}}

\newcommand{\R}{\mathbb{R}}
\newcommand{\Sp}{\mathrm{Sp}}
\newcommand{\SP}{\mathrm{SP}}
\newcommand{\TC}{\mathop{\mathrm{TC}}}
\newcommand{\cx}{\mathop{\mathrm{cx}}}
\newcommand{\tc}{\mathop{\mathrm{tc}}}

\newcommand{\U}{\mathrm{U}}

\newcommand{\Co}{\Sigma}

\begin{document}

%\title{Insert your title here%\thanks{Grants or other notes

\title[Generalized motion planning]{Homotopic distance and generalized motion planning}%


%\titlerunning{Generalized motion planning}        % if too long for running head

\author{%
	E.~Mac\'{\i}as-Virg\'os    \and
        D.~Mosquera-Lois \and
        M.J.~Pereira-S\'aez %etc.
}


\address{%
		E. Mac\'{\i}as-Virg\'os\\
              Departamento de Matem\'aticas, Universidade de Santiago de Compostela, 		15782-SPAIN}
              \email{quique.macias@usc.es}          
		
 \address{%
           	D.~Mosquera-Lois \\
              Departamento de Matem\'aticas, Universidade de Santiago de Compostela, 15782-SPAIN}
               \email{david.mosquera.lois@usc.es}
		
\address{%
           	M.J.~Pereira-S\'aez\\
              Facultade de Econom\'{\i}a e Empresa, Universidade da 			Coru\~na, 15071-SPAIN}
              \email{maria.jose.pereira@udc.es}

\thanks{The first and third authors were partially supported by MINECO Spain research project MTM2016-78647-P. The first author was partially supported by Xunta de Galicia ED431C 2019/10 with FEDER funds. The second author was partially supported by Ministerio de Ciencia, Innovaci\'on y Universidades, grant FPU17/03443  and Xunta de Galicia ED431C 2019/10 with FEDER funds.}
\maketitle

%\setcounter{tocdepth}{2}
%\tableofcontents
%
%\bigskip

\begin{abstract}
%Insert your abstract here. Include keywords, PACS and mathematical
%subject classification numbers as needed.
We prove that the homotopic distance between two maps defined on a manifold is bounded above by the sum of their subspace distances on the critical submanifol of any Morse-Bott function. This generalizes the Lusternik-Schnirelmann theorem (for Morse functions), and a similar result by Farber  for the topological complexity. Analogously, we prove that, for analytic manifolds, the homotopic distance is bounded by the sum of the subspace distances on any submanifold and its cut locus.  As an application, we show how navigation functions can be used to solve a generalized motion planning problem.


\end{abstract}

\section{Introduction}

%\subsection{}

Both Lusternik-Schnirelmann category (\cite{CLOT2003}) and Farber's topological complexity (\cite{Fa2003}) can be seen as particular cases of the {\em homotopic distance} between maps, introduced by  the authors in \cite{MM2020}.

The importance of those homotopy invariants  is well known. On one hand, the L-S category of a compact differentiable manifold  gives  a lower bound to the number of critical points of any differentiable function defined on it. On the other hand, the topological complexity  is closely related to  the problem of designing robot motion planning algorithms for a given configuration space. 

In fact, the two  approaches above are connected. As noted by Farber (\cite{Fa2008}), following previous ideas by Koditschek and Rimon (\cite{KR1990}), the negative gradient vector field of a Morse-Bott real valued 
function gives rise to a flow which
moves any given initial condition towards a target critical point. Therefore, so-called {\em navigation functions}  provide motion planning algorithms for moving from
an arbitrary source to an arbitrary target.

It has been repeatedly  observed that  topological complexity shares many   properties with L-S category, and that both invariants lead to similar results. Examples are   formulas involving products, fibrations and cohomological bounds. As shown in \cite{MM2020}, the reason for this phenomenon is that those results can be proven for the homotopic distance between maps.

A paradigmatic result is Lusternik-Schnirelmann's Formula \eqref{LS-FORMULA}, that relates, for a given differentiable function $\Phi\colon M\to \mathbb{R}$, the L-S category of the ambient manifold $M$ with the subspace L-S category of the critical levels of $\Phi$.  While the original proof of this result for arbitrary differentiable functions involves the subtle mini-max principle which is at the heart of  L-S theory, the proof for Morse-Bott functions is much easier and only needs the most basic properties  of   L-S category.
The analogous Formula \eqref{TCFORMULA} for the topological complexity was proved by Farber in \cite[Theorem 4.32]{Fa2008}.

In this paper we shall prove a similar result for the homotopic distance between two maps (Theorem \ref{MORSEBOTT}). This more general formula can then be applied to other invariants such as the {\em topological complexity of a work map}, introduced by Murillo -- Wu (\cite{MW2019}) and  Scott (\cite{Sc2020}), the {\em complexity of a fibration}, defined   by Pavesic in \cite{Pa2017},  and the {\em weak topology} of Yokoi (\cite{Y2018}), which are particular cases of the homotopic distance too.

Finally, we adapt our result to the Morse-Bott function given by the (square of) the distance to a submanifold $N$ in a Riemannian manifold $M$.  We obtain (Theorem \ref{MAINCUT}) that the  homotopic distance between two continuous maps on the manifold is bounded by the sum of the  subspace homotopic distances on the submanifold $N$ and its {\em cut locus} $\Cut N$. Our proof assumes that the manifold is analytic, in order to guarantee that the cut locus is triangulable.

%\subsection{}
The idea of using Morse-Bott functions to estimate the LS category of some homogeneous spaces comes back to Kadzisa and Mimura (\cite{KM2011}). However, they did not use Formula \eqref{LS-FORMULA}, but instead they constructed  conedecompositions of the
manifold by  the  gradient flows. This gives a cone length, which
is an upper bound of the LS category (\cite[Section 3.5]{CLOT2003}).

In  \cite{MP2013}, the authors used Formula \eqref{LS-FORMULA} directly, once a convenient function  was chosen, to give the upper bound $\cat \Sp(n)\leq (n+1)n/2$ for the L-S category of the
symplectic group. In a similar way, an optimal upper bound was given in \cite{MPT2017} for the LS category of the quaternionic Grassmannians $G_{n,k}=\Sp(n)/(\Sp(k) \times \Sp(n - k))$.
%, which is known to be $k(n - k)$.

On the other hand, Farber (\cite{Fa2008}) and later Costa (\cite{Co2010}) used Formula \eqref{TCFORMULA} to study {\em navigation functions} on the torus $T^n$, the projective spaces $\R P^n$ and the lens spaces $L(p,q)$, thus bringing new light
into known results regarding topological complexity.

We shall define a {\em general motion planning} problem, meaning that, given two maps $f,g\colon X \to Y$, we need to find,  for each $x\in X$,  a path $s(x)$ on $Y$, depending continuously on $x$ and connecting the points $f(x)$ and $g(x)$. 

This problem can also be  solved  with navigaton functions.  
Navigation functions  
exploit the gradient flow of a Morse-Bott function for constructing motion
planning algorithms.  
Originally, Koditschek and Rimon  (\cite{KR1990}) studied machines that navigate to a fixed goal using a gradient flow technique. Later, Farber (\cite{Fa2008}) considered navigation functions  
which depend on two variables, the source and the target.   We shall adapt his explanation to our generalized setting.

The contents of the paper are as follows. In Section \ref{DOS} we recall the basic definitions of Morse-Bott function, Lusternik-Schnirelmann category and Farber's topological complexity, as well as the classical theory relating the latter two invariants with the critical submanifolds of a Morse-Bott function. In Section \ref{TRES} we recall the definition of {\em homotopic distance} between two continuous maps, introduced by the authors in \cite{MM2020}, and we give a {\em subspace} or {\em relative} version of it (Definition \ref{SUBDIST}). This notion generalizes the   subspace L-S-category (\cite[Definition 1.1]{CLOT2003}) and the relative topological complexity   (\cite[Section 4.3]{Fa2008}). We prove its homotopic invariance  in Proposition \ref{prop:homot_invariance}.
In Section \ref{CUATRO} we study how to compute the homotopic distance between two continuous maps defined on a manifold $M$ on which it is also defined a Morse-Bott function $\Phi$, by reducing the computation to the critical levels of $\Phi$. In order to do that, we first show how to modify our definition of subspace homotopic distance to deal with  {\em Euclidean neighbourhood retracts} (ENRs)   instead of open subsets. Our main result (Theorem \ref{MORSEBOTT}) states that the  homotopic distance on $M$ is bounded above by the sum of the  subspace distances on the critical levels. This result generalizes both the classical Lusternik-Schnirelmann theorem (\cite{LS1934})  and the analogous Farber's result for the topological complexity (\cite[Theorem 4.32]{Fa2008}), whose proof we have adapted to our context. 

In Section \ref{CINCO} we consider a complete Riemannian manifold $M$, a submanifold $N$ and the function given by the (square of) the distance  to $N$. This function turns to be differentiable on  $M\setminus \Cut N$, where $\Cut N$ is the {\em cut locus} of  $N$. We have given a quick survey of the main properties of the cut locus, in order to prove that the homotopic distance between two maps defined on a compact  analytic manifold is bounded by the sum of the subspace homotopic distance on the submanifold $N$ and the subspace homotopic distance on its cut locus $\Cut N$ (Theorem \ref{MAINCUT}). An obvious consequence for the L-S category is (Corollary \ref{CORCAT}):
$$\cat M \leq \cat\nolimits_M(N) + \cat\nolimits_M(\Cut N) +1.$$
Finally, we show in Sections \ref{MOTPLAN} and \ref{NAVIG} how to interpret the preceding results in terms of {\em navigation functions} that solve a {\em  generalized motion planning}. In fact, following the original interpretation of the topological complexity, it happens that the homotopic distance $\Di(f,g)$ between two continuous maps $f,g\colon X \to Y$ measures the difficulty  of finding, for a given $x\in X$,  a continuous path in $Y$ that  connects the points $f(x)$ and $g(x)$. When $X=M$ is a manifold, and there is a Morse-Bott function $\Phi$ defined on it, Theorem \ref{MORSEBOTT} can be interpreted as follows: first, we solve the subspace motion planning problem on each critical submanifold $\Sigma_i$ by means of a covering $G_j^i$ by ENRs. Then, if $x\in M$ belongs to the {\em bassin of atraction} $V_j^i$ of $G_j^i$, we slide along the gradient flow $x(t)$  from $x=x(0)$ to a critical point $\alpha=x(\infty)\in G_j^i$. Since there is a path $\gamma$ connecting $f(\alpha)$ and $g(\alpha)$, we can concatenate the paths $f(x(t))$,$\gamma$ and $g(\bar x(t))$, where $\bar x(t)$ is the reverse path of $x(t)$.

A similar interpretation  is valid for Theorem \ref{MAINCUT}, because there is a Morse-Bott flow collapsing $M\setminus \Cut N$ to $N$.

Finally, we show that these results   apply  not only to L-S category and topological complexity, but also to other invariants like the {\em naive topological complexity} $\tc(f)$ of the work map $f$,
studied by Farber (\cite[p. 5]{Fa2008}) and later by Murillo and Wu (\cite{MW2019}) and by
Scott (\cite[Theorem 3.4]{Sc2020});
the {\em topological
complexity $\cx(f)$ of a fibration} $f$, defined by Pavesic (\cite{Pa2017}); and the {\em weak category} of a continuous map 
$f\colon X \to X$, defined by
Yokoi  (\cite{Y2018}).
\medskip

All along this paper we assume that  manifolds and topological spaces are path-connected, unless otherwise stated.

\section{Basic definitions}\label{DOS}

We begin by recalling the basic facts and notations of Morse-Bott theory. 
Also we recall the  definitions of L-S category and topological complexity.

\subsection{Morse-Bott theory}
Let $M$ be a compact differentiable manifold. The smooth function $\Phi\colon M \to \R$ is called a {\em Morse-Bott function} if the critical set $\crit \Phi$  is 
a disjoint union of connected  submanifolds $\Co$ and for each critical point $p\in \Sigma \subset  \crit \Phi$   the Hessian 
is non-degenerate in the  directions transverse to $\Sigma$. A general reference is Nicolaescu's book (\cite{Ni2007}). For a complete proof of the Morse-Bott Lemma about the local structure of such a function see Banyaga-Hurtubise's paper \cite{BH2004}.

 If
$p$ is a critical point of  $\Phi$, the {\em index} of $p$  is the number of negative eigenvalues of the Hessian $(\Hess \Phi)_p$. This number is constant along any connected critical submanifold $\Co$. If $\varphi\colon M \times \R \to M$ denotes the negative gradient flow of $\Phi$, 
the {\em stable manifold} or {\em basin of attraction} of  $\Co$ is the set
%NICOLAESCU
$$S(\Co) = \{p\in M\colon \lim_{t\to +\infty} \varphi(p,t)\in \Co\}.$$
It is well known that the map
$$\pi\colon S(\Co) \to \Co,$$
sending each point $p$ to the limit point of its trajectory,
is a  fibre bundle with fiber $\R^{m-n-k}$, where $m=\dim M$, $n=\dim \Co$ and $k$ is the index of $\Co$. Moreover, $M$ as a set is the disjoint union of the submanifolds $S(\Co)$. Notice, however, that the global limit map $M \to \crit\Phi$ is {\em not} continuous.

\subsection{L-S category}
 A fundamental references for L-S category  is   \cite {CLOT2003}.
%and topological complexity areis \cite {Fa2003}, %respectively.

Let $X$ be a topological space. A subspace $A\subset X$ is {\em $0$-categorical in $X$} if it can be contracted to a point inside $X$. Equivalently, the inclusion map $A\subset X$ is homotopic to a constant map.

More generally, 
\begin{definition}\label{LSCATEG}
The {\em Lusternik-Schnirelmann} category of $A$ in $X$, denoted by $\cat_X A$,  is the minimum integer $k\geq 0$ such that there is a covering  $U_0\cup\dots\cup U_k=A$, with the property that
each subset $U_j$ is open in $A$. 
\end{definition}

If such a covering does not exist, we write $\cat_X A=\infty$. When $A=X$, we simply write $\cat_X X=\cat X$.

 There is a well known relationship between L-S category and the number of critical points.
The following more elaborated result  already appeared in the 
Lusternik-Schnirelmann's original  work  \cite{LS1934}. We heard about it for the first time in Rudyak-Schlenk's paper  \cite{RS2003}, see also Reeken's work  \cite[p.~21]{Re1972}.

\begin{theorem}\label{LS-CRIT} Let $M$ be a compact smooth manifold. Let $\Phi\colon M \to \R$ be a smooth function with critical values $c_1<\dots<c_p$, and let $\Sigma_i= \Phi^{-1}(c_i)\cap \crit \Phi $ be the set of critical points which lie in the level $\Phi=c_i$. Then 
\begin{equation}\label{LS-FORMULA}
\cat M +1 \leq \sum_{i=1}^p(\cat\nolimits_M \Sigma_i +1).
\end{equation}
\end{theorem}
The following result improves the latter formula by observing that different connected critical submanifolds lying on the same critical level can be aggregated. 
\begin{proposition}Let $C_i^1,\dots, C_i^{n_i}$ be the connected components of $\Sigma_i$. Then 
$$\cat\nolimits_M \Sigma_i=\max_j \cat\nolimits_MC_i^j.$$
\end{proposition}

 \subsection{Topological complexity}
 The fundamental reference for topological complexity is \cite{Fa2003}.
 
 The following definiton is the original one in Farber's paper  \cite[Definition 4.20]{Fa2008}, although we have normalized it. Also, we  have chosen  to say ``{\em subspace} topological complexity'' instead of ``{\em relative} topological complexity''.
 
 \begin{definition}
\label{TOPCOMPLEX}Let $X$ be a topological space and $A \subset X \times X$ be a
subspace. The (normalized) {\em subspace topological complexity}, denoted by $\TC_X(A)$, is 
 the smallest integer $k\geq 0$ such that there
is a cover $U_0 \cup \cdots \cup U_k= A$ with the property that each
$U_j \subset A$ is open in $A$,
and the projections  $U_j \rightrightarrows X$ on the first and the
second factors are homotopic to each other.
 \end{definition}

We simply write $\TC_X(X\times X)=\TC(X)$.\\
 
 The following result is due to Farber (\cite[Theorem 4.32]{Fa2008}).

\begin{theorem}\label{TC-CRIT}Let $M$ be a compact smooth manifold (without boundary). Let $\Phi\colon M \times M \to \R$ be a
Morse-Bott  function such that $\Phi\geq 0$ and $\Phi(x,y)=0$ if and only if $x=y$. Then
\begin{equation}\label{TCFORMULA}
\TC(M)+1\leq \sum_{i=1}^p \left(\TC\nolimits_M(\Sigma_i)+1\right),
\end{equation}
where $\Sigma_1,\dots,\Sigma_p$ are the critical levels of $\Phi$.
\end{theorem}

 Farber also proves that $\TC_M(\Sigma_i)$ equals the maximum of the subspace topological complexities of the connected components.

A function like that of Theorem \ref{TC-CRIT} is called  by Farber a {\em navigation function}  (cf. Section \ref{NAVIG}).

\section{Homotopic distance}\label{TRES}
The following notion was introduced by the authors in \cite{MM2020}.

\begin{definition}\label{PRIMERA}Let $f,g\colon X \to Y$ be two continuous maps. The {\em homotopic distance} $\Di(f,g)$ between $f$ and $g$  is the least integer $n\geq 0$ such that there exists an open covering $U_0\cup\cdots \cup U_n=X$  with the property that the restrictions  $f_{\vert U_j}$ and $g_{\vert U_j}$ are homotopic maps, for all $j=0,\dots,n$.  

If there is no such covering, we define $\Di(f,g)=\infty$.
\end{definition}

\begin{example} 
Let   $X$ be a path-connected topological space. The   L-S-category $\cat X$   of  $X$ equals the homotopic distance between the identity $\id_X$   and any constant map,   $\cat(X)=\Di(\id_X,x_0)$.
\end{example}

\begin{proposition}\cite[Proposition 2.5]{MM2020}
Given a base point $x_0\in X$ we define the axis inclusion maps $i_1,i_2\colon X\to X\times X$ as $i_1(x)=(x,x_0)$ and $i_2(x)=(x_0,x)$.  
The homotopic distance between $i_1$ and $i_2$ equals the LS-category of $X$, that is, $\Di(i_1,i_2)=\cat(X)$.
\end{proposition}

\begin{example}
More generally, the L-S category of a map $f\colon X \to Y$ (\cite[Exercise 1.16, p.~43]{CLOT2003}) is the distance between $f$  and any constant map, $\cat f =\Di(f,x_0)$, when $Y$ is path-connected. For instance, the category of the diagonal $\Delta_X\colon X \to X \times X$ equals $\cat X$. \end{example}

Let $\Pa X=X^I$ be the path space of $X$ and let $\pi\colon \Pa X \to X\times X$, $\pi(\gamma)=\big(\gamma(0),\gamma(1)\big)$, be the path fibration sending each continuous path $\gamma\colon [0,1]\to X$  to its initial and final points.

\begin{proposition}\cite[Proposition 2.6]{MM2020}
  The  topological complexity $\TC(X)$ of the topological space $X$ equals the homotopic distance between the two projections $p_1,p_2\colon X\times X \to X$, that is,
	$\TC(X)=\Di(p_1,p_2)$.
\end{proposition}

Other examples will be  given later (see Section \ref{MOTPLAN}).

\medskip{}

We propose the following definition of  {\em subspace distance}, that is, homotopic distance between two maps with respect to a subspace, as a generalization of Definitions \ref{LSCATEG} and \ref{TOPCOMPLEX}.

\begin{definition}\label{SUBDIST}Let $f,g\colon X \to Y$ be two continuous maps, and let $A\subset X$ be a subspace. The {\em subspace distance} between the two maps $f,g$ on $A$, denoted by $\Di_X(A;f,g)$, is defined as the distance  between the restrictions of $f,g$ to $A$, that is,
$$\Di_X(A;f,g):=\Di(f_{\vert A}, g_{\vert A}).$$
\end{definition}

Obviously, when $A=X$ we recover the usual homotopic distance. Moreover, observe that $\Di_X(A;f,g)=\Di(f\circ i_A,g\circ i_A)$, where $i_A\colon A \subset X$ is the inclusion.

\begin{example}If $i_A\colon A\subset X$ is a subspace, then 
$$\cat\nolimits_X A=\Di_X(A;\id_X,x_0)=\Di(i_A,x_0).$$
\end{example}

\begin{example}
If $i_A\colon A\subset X\times X$ is a subspace, and $p_1,p_2\colon X \times X \to X$ are the projections, then
 $$\TC\nolimits_X(A)=\Di_{X\times X}(A; p_1,p_2).$$
\end{example}

The following result will be used later.
 \begin{proposition}
Let $f,g\colon X \to Y$ be two continuous maps, and let $\{A_i\}_{i=1}^n$ be the connected components of $X$. Then 
$$ \Di(f,g)=\Di_X(X;f,g)=\max_i\Di_X(A_i;f,g).$$
\end{proposition}
\begin{proof}
For the simplicity of the notation, we will do the proof for the case of two connected components $A_1$ and $A_2$. Say $\Di_X(A_1;f,g)=n_1\leq \Di_X(A_2;f,g)=n_2$ with open coverings $\{U_i\}_{i=0}^{n_1}$ and $\{V_i\}_{i=0}^{n_2}$ of $A_1$ and $A_2$, respectively. Then $\{U_i\cup V_i\}_{i=0}^{n_1} \cup \{V_i\}_{i=n_1+1}^{n_2}$ is an open cover of $X$. Notice that $U_i\cap V_i=\emptyset$ for $i=0,\dots,n_1$, thus guarantying that $f$ and $g$ are homotopic on $U_i\cup V_i$. Hence,  $\Di_X(X;f,g)\leq n_2$.

The other inequality, $n_2=\Di(A_2;f,g)\leq \Di_X(X;f,g)$ simply follows from $A_2\subset X$.
\end{proof}
The main property of the homotopic distance, and in consequence of $\cat$ and $\TC$, is its homotopy invariance (\cite[Proposition 3.13]{MM2020}). We shall need the following relative version.

\begin{proposition}\label{prop:homot_invariance}Let $f,g \colon X\to Y$  be two continuous maps. Let $i_A\colon A \hookrightarrow X$ and $i_B\colon B \hookrightarrow X$ be two subspaces, and let $\alpha \colon A \to B$ be a homotopy equivalence, such that $i_B\circ\alpha \simeq i_A$:
$$\xymatrix{
%1
A\ar@<+0.0ex>[d]_{\ \alpha\ }\ar@<+0.0ex>[r]^{\ i_A\ }&X\ar@<+0.5ex>[r]^{\ f\ }\ar@<-0.5ex>[r]_{\ g\
} &\ Y\\
%2
B\ar@<+0.0ex>[ur]_{\ i_B\ }&
} $$
Then
$\Di_X(A;f ,g )=\Di_X(B;f,g)$.
\end{proposition}
\begin{proof} 
Since $\alpha$ is a homotopy equivalence 
and $i_B\circ \alpha \simeq i_A$ (cf. \cite[Propositions 2.2 and 3.12]{MM2020}) we have
$$\begin{array}{ll}
\Di_X(B; f,g)&=\Di(f_{\vert B},g_{\vert B})=\Di(f\circ i_B,g\circ i_B)=\Di(f\circ i_B\circ\alpha,g\circ i_B\circ\alpha)\\
&= \Di(f\circ i_A ,g\circ i_A)=\Di(f_{\vert A},g_{\vert A})=\Di_X(A; f, g).\quad%\qedhere
\end{array}$$
\end{proof}
Finally, it is easy to prove the following sub-additivity property:
\begin{proposition}\label{lema:subadditivity}
		Given  maps $f,g\colon X\to Y$ and a finite open  covering $V_1\cup \cdots \cup V_p=X$, 
		it happens that
		$$\Di(f,g)+1\leq \sum_{i=1}^p\left(\Di_X(V_i;f ,g )+ 1\right).$$
	\end{proposition}
	
	\section{Homotopic distance and Morse-Bott functions}\label{CUATRO}
	
	In order to generalize Theorems \ref{LS-CRIT} and \ref{TC-CRIT}, we need to tackle Definition \ref{PRIMERA} in a situation that does not  demand that the pieces in which we decompose the space are open subsets. In fact, for a Morse-Bott function, the pieces will be the basins of the negative gradient flow, which are submanifolds, but neither open nor closed subspaces, in general.

Hence, we shall restrict ourselves to smooth manifolds and submanifolds, or more generally, to topological spaces which can be enlarged to an open neighbourhood.
These are the so-called {\em Euclidean neighbourhood retracts} (ENR, for short).

\begin{definition}\cite[p. 448]{Di2008}, \cite[p. 81]{Do1972} A topological space $E$ is called a {\em Euclidean neighbourhood retract} (ENR for short) if it is homeomorphic to a subspace $E'\subset \R^n$ which is a retract of some neighbourhoud $E'\subset W \subset \R^n$.
\end{definition}

The class of ENRs includes all finite-dimensional cell
complexes and all compact topological manifolds (\cite[Appendix E]{Br1993}).

We need the following property.

\begin{proposition}\label{ENRS}\cite[Cor. 8.7]{Do1972}, \cite[Remark 18.4.4]{Di2008}
Let $A\subset X$ be two ENRs. Then  
there exists an open neighborhood $A\subset U\subset X$ of $A$ in $X$  and a retraction $r \colon U \to A$
such that the inclusion $i_U\colon U \hookrightarrow X$ is homotopic to $i_A\circ r$, where $i_A\colon A \hookrightarrow X$ denotes the inclusion. That is, we have a diagram, commutative up to homotopy,
$$\xymatrix{
 &\ U\  \ar[d]^r\ \ar@<+0.5ex>[dr]^{i_U\
} &\\
&A\ \ar@<-0.0ex>[r]_{\ i_A\
}&X} 
$$
\end{proposition}

\begin{corollary}\label{TUBE}Let $A\subset X$ be two ENRs, and let $f,g\colon X  \to Y$  be two continuous maps. If $f,g$ are homotopic on  $A$, then there is an open neighbourhood $A\subset U \subset X$ such that
$f,g$ are homotopic on $U$.
\end{corollary}

\begin{proof}Let    $r\colon U \to A$ as in Corollary \ref{ENRS}. We have
$$f_{\vert U}=f\circ i_U\simeq f\circ i_A \circ r =f_{\vert A}\circ r \simeq g_{\vert A}\circ r =g \circ i_A \circ r \simeq g\circ i_U=g_{\vert U}.$$
\end{proof}

\begin{corollary}$\Di_X(A;f,g)=\Di_X(U;f,g)$.
\end{corollary}
\begin{proof}
It is an inmediate consequence of Proposition \ref{prop:homot_invariance} and Proposition \ref{ENRS}.  
\end{proof}

 With the previous ingredients, we are ready to state the following result.

\begin{theorem}\label{MORSEBOTT}Let $\Phi\colon M \to \R$ be a Morse-Bott function in the compact smooth manifold $M$. Let $c_1<\dots< c_p$ be its critical values, and let $\Sigma_i=\Phi^{-1}(c_i)\cap \crit \Phi$ be the set of critical points in the level $ \Phi=c_i$.  If $f,g\colon M \to Y$ are two continuous maps, then
$$\Di(f,g)+1\leq \sum_{i=1}^p(\Di_M(\Sigma_i;f,g)+1).$$
\end{theorem}

\begin{proof}For each $\Sigma_i$, let $S(\Sigma_i)$ be the basin $\pi^{-1}(\Sigma_i)$ of  points in $M$ whose limit point $\varphi(p,\infty)\in \Sigma_i$. Since $\pi\colon S(\Sigma_i ) \to \Sigma_i$ is a homotopy equivalence, by Proposition \ref{prop:homot_invariance} we have
$$\Di_M(\Sigma_i;f,g)=\Di_M(S(\Sigma_i);f,g).$$
Now, each $S(\Sigma_i)$ is a submanifold of $M$, so by Proposition \ref{ENRS},
there is an open subset $\Sigma_i\subset V_i \subset M$ such that
$$\Di_M(S(\Sigma_i);f,g)=\Di_M(V_i;f,g).$$
Finally, since $\cup_i S(\Sigma_i)=M$, we have $\cup_i V_i=M$, and by Proposition \ref{lema:subadditivity} we have
$$\Di(f,g)\leq \sum_i \Di_M(V_i;f,g)=\sum_i \Di_M(\Sigma_i;f,g),$$
as stated.
\end{proof}

\section{Cut locus}\label{CINCO}
In Theorem \ref{MORSEBOTT} (and in Theorems \ref{LS-CRIT} and \ref{TC-CRIT}, which are particular cases), we gave an upper bound for the homotopic distance between two continuous maps $f,g$, defined on a manifold $M$, by considering their restrictions to the critical set of a differentiable Morse-Bott function $\Phi$ also defined on $M$.

In what follows, we shall consider a submanifold $N$ of a  complete Riemannian manifold $M$ and the function $\Phi\colon M\to \mathbb{R}$, $\Phi(y)=d(y,N)^2$ given by (the square of) the distance from the point $y\in M$ to the submanifold $N$. It is well known that this function is not differentiable in the {\em cut locus} of $N$ (\cite[Proposition 4.8]{Sakai}). 
%$\Cut N$ 
Anyway, we shall try to adapt our preceding results to this setting.

\subsection{Preliminaries}
We begin by recalling some basic facts about the cut locus. See for instance \cite[Chapter VIII.7]{KN1967} or \cite[Chapter II.C]{GHL1990}. Good surveys are  Kobayashi's paper \cite{Ko1989} or the more recent Angulo's thesis \cite{An2014}.
%We follow Itoh-Tanaka's paper \cite{IT2000} for the first definitions. 
%y Berceanu \cite{B}
%Gluck-Singer
%Kobayashi-Nomizu
%Gallot-Lafontaine

Let $M$ be a complete ${\Class}^\infty$  Riemannian manifold and let  $N$ be an immersed submanifold. The {\em Riemannian distance} $d(y,N)$, from $y\in M$ to $N$, is the infimum  of the lengths of all
piecewise smooth curves joining $y$ to some point $x\in N$. 

A geodesic $\gamma$ between two points $x,y\in M$ is said to be minimizing if its length  equals  the distance $d(x,y)$. 
A unit speed geodesic $\gamma\colon [0,t_0]\to M$ emanating from  $N$  is an {\em $N$-segment} (or $N$-minimizing geodesic) if its length $t$ equals the distance $ d(\gamma(t),N)$,  for all $t\in [0,t_0]$ (\cite{IT2001}). In this case the geodesic must be ortohogonal to $N$. 

\begin{definition}
 The point $\gamma(t_0)$ is called a {\em cut point} of $N$ if there is no  $N$-segment properly containing $\gamma([0,t_0])$. The cut locus $\Cut N$   is the set of all these cut points.
 \end{definition}

The simplest case is when $N=\{x\}$ is a unique point. On each geodesic curve emanating from the point $x$,  the cut point is the last point to which the geodesic minimizes distance.

\begin{example} If $M=S^n$ is the $n$-dimensional unit sphere and $x$ is its
North Pole,  the cut locus of $x$ reduces
to the South Pole. If $M= \R P^n$ is the projective space, the
Riemannian metric of $S^n$ induces a Riemannian metric on  $M$, 
so that the projection of $S^n$ onto $M$ is a local
isometry. The cut locus of the point $x$ corresponding to the North and South Poles of the sphere $S^n$ is the
image of the equator of $S^n$ under the projection, that is a naturally imbedded $(n - 1)$-dimensional
projective space $\R P^{n-1}$. Analogously, in $\C P^n$ the cut locus of a point $x$ is isometric to $\C P^{n-1}$.
\end{example}

\begin{example}\label{EXTORUS}
If $M=T^2=S^1\times S^1$ is the torus seen as a quotient of the square $I\times I$, the usual Riemannian metric on $\R^2$ induces a Riemannian metric on $M$. The cut locus of the point $x=[(1/2,1/2)]$ can be identified with the image of the boundary of the square, that is,
\begin{equation}\label{CUTPROD}
\Cut (x) =(C \times S^1) \cup (S^1 \times C)=S^1\vee S^1,
\end{equation}
where $C=\{[0]\}$ is the cut point of $[1/2]$ in the circle $S^1$ seen as a quotient of $I$.

\end{example}

\subsection{Structure of the cut locus}

When $N$ is an arbitrary submanifold, we denote by $\pi\colon (\U\nu) N\to N$ the unit normal bundle to $N$. Let  $\gamma_u(t)$ be the unit speed geodesic emanating from $x\in N$ in the direction of $u\in (\U\nu)_x N$. The transverse exponential map $\Exp_x(tu)=\gamma_u(t)$ is a diffeomorphism from a tubular neighborhood
of the zero section of the normal bundle $\nu N$ of $N$ into a tubular neighborhood
of $N$ in $M$, but singularities can appear for large vectors. 
For a vector $w=tu\in \nu_x N$ 
where $\Exp_x$ is not regular, the {\em order of conjugacy} of $w$ is the dimension $k>0$ of the kernel of the linear
map $D_w\Exp_x$. The point $y=\Exp_x(w)$ is called a {\em conjugate point} of $x\in N$.

It is standard that the points of $\Cut N$ are either the first conjugate  point $y\in M$ on a length minimizing geodesic starting at $N$, or a ``separation point'', that is, a point $y\in M$ where there are at least two length minimizing geodesics from $N$ to $y$. 
On a simply connected complete symmetric space, the cut locus of a point coincides with the  first conjugate locus  (\cite[Theorem 5]{Cr1962}).

\begin{proposition}\cite[Theorem 3.26]{BP2020}\label{CLOSED}
If the manifold $M$ is complete, the cut locus $\Cut N$ of a compact submanifold $N$ is a closed subset of $M$; in fact, it is the closure of the separation points. 
 \end{proposition}

In practice, the cut locus is very hard to compute, since in general it has a wild structure, like a stratified manifold. For instance, the usual metric on the sphere $S^n$ can be deformed around the equator in such a way that the North Pole has a non-triangulable cut locus \cite[Theorem A]{GS1978}.
%In fact, for any smooth manifold of dimension $n\geq 2$ and any point $x\in M$, there is a Riemannian metric with non-triangulable cut locus. 
  In general, the cut locus of a surface can have branch points.
  % or it can be a non-triangulable space. 
  
  On the other hand, in any compact manifold $M$ there exists some metric such that the cut locus of any point is triangulable (\cite[p. 348]{GS1978}).

 Fortunately, the situation is much better in  Lie groups  and homogeneous spaces, endowed with their natural structures and metrics, which are known to be analytic manifolds. 
  
  \begin{theorem}[Buchner]
(  \cite{Bu1977},\cite[Theorem 3.9]{BP2020}) Let $M$ be an analytic manifold of dimension $m$, and let $N$ be an analytic submanifold. Then the cut locus $\Cut N$   is a simplicial complex of dimension strictly less than $m$.  
  \end{theorem}
  
\begin{example}
 (\cite[Theorem 3.23]{BP2020}) The cut locus of a point in a real analytic closed orientable surface of genus $g$ is  a connected graph, homotopically equivalent to a wedge of $2g$ circles (cf. the torus in Example \ref{EXTORUS}).
 \end{example}

\subsection{The distance function}

The result we are interested in is the following theorem (\cite[Theorem 3.28]{BP2020}):

\begin{theorem}\label{BUCHNER} Let $N$ be a closed embedded submanifold of a complete Riemannian manifold $M$. Let
$d \colon M \to \R$ be the distance function with respect to $N$. The restriction of the function $\Phi=d^2$  to $M \setminus \Cut N$ is
a Morse-Bott function, with $N$ as the critical submanifold. As a consequence, the gradient flow deforms $M\setminus \Cut N$ to $N$.
\end{theorem}

Hence, we have the following formula.

\begin{theorem}\label{MAINCUT}Let $M$ be a compact analytic  manifold and $N$ a closed analytic submanifold. Let $f,g\colon  M \to Y$ be two continuous maps. Then
$$D(f,g) \leq D_M(N; f,g) + D_M(\Cut N; f,g)+1.$$
\end{theorem}

\begin{proof}It is known that a finite simplicial complex is an ENR \cite[Corollary A.8]{Ha2002}.  
Hence  $\Cut N$ is an ENR, by Theorem \ref{BUCHNER}. By Corollary \ref{TUBE} there is an  open neighbourhood $U$ such that $\Di_M(\Cut N; f,g)=\Di_M(U;f,g)$. Also, we now that $M\setminus \Cut N$ is open in $M$, by Proposition \ref{CLOSED}. Hence
$$M=(M\setminus \Cut N)  \cup \Cut N  =(M\setminus \Cut N) \cup U,$$
so by the subaditivity property (Proposition \ref{lema:subadditivity}) we have 
$$\Di(f,g)\leq \Di_M(N;f,g)  + \Di_M(\Cut N;f,g) +1,$$
because $M\setminus \Cut N$ and $N$ are homotopically equivalent in $M$, so Proposition \ref{prop:homot_invariance} applies.
\end{proof} 

\begin{corollary}\label{CORCAT}
By taking the identity and a constant function, we have 
$$\cat\nolimits M \leq \cat\nolimits_M(N) + \cat\nolimits_M(\Cut N).$$
\end{corollary}

Now, in order to have an analogous result for the topological complexity, we can take the projections $p_1,p_2\colon M \times M \to M$. 
\begin{corollary}
$\TC(M)\leq  \TC_M(N\times N) + \TC_M(\Cut (N \times N)+1$.
\end{corollary}

This result may be seen as a formalization of the following comment by  Blaszczyk and Carrasquel in \cite{BC2018}: ``in order to estimate topological
complexity of $X$, it is enough to understand how to motion plan between
points $p, q\in X$ with $q \in\Cut(p)$''.

Unfortunately, we do not know any explicit formula for the cut locus of $N \times N$ in  $M\times M$. For a point $N=\{x\}$, the following formula is proven by Crittenden in \cite[p. 328]{Cr1962} (compare with example \ref{EXTORUS}):
$$\Cut\nolimits_{M\times M} ((x,x)) =(\Cut\nolimits_M(x)\times M) \cup (M \times \Cut\nolimits_M(x)).$$
\begin{remark}
Notice that in general it is {\em not} true that $\TC(M)$ is bounded by $\TC_X(N)+\TC_X(\Cut N)+1$. For instance, take $M=S^n$ for $n$ even, and as submanifold $N$ the north pole $\NP$. Since $\Cut(\NP)=\{\SP\}$, the south pole, we have $\TC_{ X}(\NP)+\TC_{ X}(\SP)+1=1$ but $\TC(S^n)=2$ (\cite[Theorem 8]{Fa2003}). \end{remark}

\section{Motion planning}\label{MOTPLAN}

Both L-S category and topological complexity  measure the difficulty of finding continuous  motion planning algorithms for the configuration space $X$ of a mechanical device. We can interpret the homotopic distance between two maps $f,g\colon X \to Y$ in the same way, because it solves the following:

\medskip

{\em generalized planning problem}: given an arbitrary point $x\in X$ find a continuous path $s(x)$, joining $f(x)$ and $g(x)$ in $Y$, in such a way such that the path $s(x)$ depends continuously on $x$. 

\medskip

More precisely, let $\mathcal{P}(f,g)$ be the space of pairs $(x,\gamma)$ where $x\in X$ and $\gamma$ is a continuous path on $Y$, such that $\gamma(0)=f(x)$ and $\gamma(1)=g(x)$. This space fibers over $X$, by taking the map $\pi^{*}\colon \mathcal{P}(f,g)\to X$, where $\pi^*(x,\gamma)=x$.  Notice that $\pi^*=(f,g)^*\pi\colon \mathcal{P}(f,g)\to X$ is the pullback fibration of the path fibration  $\pi\colon Y^I\to Y\times Y$, where $\pi(\gamma)=(\gamma(0),\gamma(1))$, by the map $(f,g)\colon X\to Y\times Y$:
$$\begin{tikzcd}
	P \arrow[d,"{\pi^{*}}",swap]\arrow[r] & PY\arrow[d,"{\pi}"]\\
	X \arrow[r,"{(f,g)}"] & {Y\times Y}.
	\end{tikzcd}$$

\begin{proposition}\cite[Theorem 2.7]{MM2020}\label{prop:genus_distance} The homotopic distance between the maps  $f,g\colon X\to Y$ equals the {\em Svarc genus} of $\pi^*$, that is, the minimum number $n\geq 0$ such that there exists an open covering  $U_0\cup\cdots\cup U_n=X$, where for each $U_j$ there is a continuous section $s_j\colon U_j \to \mathcal{P}(f,g)$ of the pullback fibration $\pi^*$.
\end{proposition}

This situation covers many different scenarios:

\subsection{L-S category:}
For $f=\id_X$ and $g=x_0$ a constant map,  we have $\cat X=\Di(\id_X,x_0)$. A homotopy between $\id_X$ and $x_0$ gives a continuous path $H_t(x)$ between any point $x$ and the fixed target $x_0$.

\subsection{Topological complexity:}
This was Farbers's original idea for the motion planning problem.
For the projections $p_1,p_2 \colon X\times X \to X$, one has $\Di(p_1,p_2)=\TC(X)$. A homotopy $H_t(x,y)$ between the  projections gives a path  joining two arbitrary  points $x,y$.

\subsection{Naive topological complexity of the work map:}
A more elaborated situation, studied by Farber \cite[p. 5]{Fa2008} and later by Murillo and Wu \cite{MW2019}, occurs when $X$ is the configuration space of a multi-arm robot, and $Y$ is the ``workspace'', that is, the spatial region that can effectively be attained by the
end device of the arm.  

A so-called ``work map'' $f\colon X \to Y$ (also called a {\em forward kinematic map})  assigns to each state of the configuration space the corresponding position  of
the end effector. This map is not assumed to be bijective. When implementing
algorithms which control the task performed by the robot,
the input of such an algorithm is a pair $(x,y)\in X\times X$   of
points of the workspace, and the output is a curve in the configuration space, connecting the positions $f(x)$ and $f(y)$. The corresponding invariant is called the {\em naive topological complexity} of the work map $f$, denoted by $\tc(f)$. More precisely:
\begin{definition}
Let $f\colon  X \to Y$ be a continuous map and denote by $\pi\colon P(X)\to X\times X$ the path fibration. The {\em (normalized) naive or strict topological complexity of $f$}, denoted $\tc(f)$,  is the least integer $n$ such that $X\times X$ can be covered by $n+1$ open subsets $\U_0,\dots,U_n$, such that for each $U_i$ there exists a continuous map $\sigma_i\colon U_i\to X^{I}$ satisfying $$(f\times f)\circ \pi \circ \sigma_i=(f\times f)_{\vert U_i}.$$
\end{definition}

Later on, Scott proved (\cite[Theorem 3.4]{Sc2020}) that the previous definition can be equivalently stated as follows:

\begin{proposition}
Let $f\colon  X \to Y$ be a continuous map. The    naive   topological complexity  $\tc(f)$ equals the least integer $n$ such that $X\times X$ can be covered by $n+1$ open subsets $\{U_i\}_{i=0}^n$ such that for each $U_i$ there exists a continuous map $f_i\colon U_i\to Y^{I}$ satisfying $f_i(x_0,x_1)(0)=f(x_0)$ and $f_i(x_0,x_1)(1)=f(x_1)$.
\end{proposition}

\begin{theorem}
Let $f\colon  X \to Y$ be a continuous map.  Then, the (normalized) naive topological complexity of $f$ equals the distance of the projections composed with the map $f$. That is, $\tc(f)=\Di(f\circ p_1, f\circ p_2)$. 
\end{theorem}

\begin{proof}
Let us denote by $\pi\colon Y^I\to Y\times Y$ the path fibration. Then, Scott proved (\cite[Theorem 3.4]{Sc2020}) that there is an open subset $U_i\subset X\times X$ satisfying that there exists a continuous map $f_i\colon U_i\to Y^{I}$ satisfying $f_i(x_0,x_1)(0)=f(x_0)$ and $f_i(x_0,x_1)(1)=f(x_1)$ if and only if there is a section $s\colon U_i\to (f\times f)^{*}Y^{I}$ of the pullback fibration of $\pi$ by the map $f \times f$.
$$\begin{tikzcd}
	(f\times f)^{*}Y^{I} \arrow[d,"(f\times f)^*\pi",swap]\arrow[r] & Y^I\arrow[d,"{\pi}"]\\
	X\times X \arrow[r,"{(f\times f)}"] & {Y\times Y}
	\end{tikzcd}$$
Observe that $f\times f\colon X\times X\to Y \times Y$ is just the map $(f\circ p_1, f\circ p_2)\colon X\times X \to Y\times Y$. Therefore, by Proposition \ref{prop:genus_distance}, $\tc(f)=\Di(f\circ p_1, f\circ p_2)$. 	
\end{proof}

\subsection{Topological complexity of a map:}
The latter invariant $\tc(f)$ is different from $\cx(f)$, the {\em topological
complexity of a map} $f \colon X \to Y$, defined by Pavesic \cite{Pa2017} as the sectional complexity (that is, the number of partial solutions to the motion planning problem) of the
map 
$$\pi\colon X^I \to X\times Y, \quad\pi(\gamma)=\left(\gamma(0),f(\gamma(1))\right),$$
which assigns, to each path in  the configuration space, the initial
state and the end effector position at the final state. Later \cite{Pa2019}, he modified this definition, but in a way that does not alter the original one when $f$ is a fibration. We can interpret it as a distance: 
\begin{theorem}
Let $f\colon X\to Y$ be a fibration and consider the map $\pi\colon X^I \to X\times Y$, $\pi(\gamma)=\left(\gamma(0),f(\gamma(1))\right)$ and the projections $\pi_X\colon X\times Y \to X$ and $\pi_Y\colon X\times Y \to Y$. Then, $\cx(f)=\Di(f\circ \pi_X, \pi_Y)$.
\end{theorem}

\begin{proof}
First, suppose that there exists an open subset $U$ of $X\times Y$ and a continuous section $s\colon U\to X^I$ for $\pi$. Then we define a homotopy between $f\circ \pi_X$ and $\pi_Y$, $H\colon U\times I \to Y$ given by $H(x,y,t)=f(s(x,y)(t))$. Conversely, suppose that there exists an open subset $U$ of $X\times Y$ and a homotopy $H\colon U\times I \to Y$ between $f\circ \pi_X$ and $\pi_Y$. Consider the following commutative diagram:
$$\begin{tikzcd}
	U\times \{0\} \arrow[d,"i",hook,swap]\arrow[r,"\pi_X"] & X\arrow[d,"f"]\\
	U\times I \arrow[r,"H"] & Y
	\end{tikzcd}$$
Since $f$ is a fibration, there exists a lifting $\widetilde{H} \colon U\times I\to X$ of $H$,  which makes the diagram commutative. Now, we define a section $s\colon U\to P(X)$ of $\pi$, given by $s(x,y)(t)=\widetilde{H}(x,y,t)$.	
\end{proof}

\subsection{Weak category}
Yokoi (\cite{Y2018}) defined  the following {\em weak category} of a continuous map 
$f\colon X \to X$   from a topological space $X$ into itself, with respect to a subspace $A\subset X$.
This invariant is of interest when studying a notion of discrete Conley index of an isolated invariant set.

\begin{definition}
The {\em weak category}
of $f$ reduced to $A$, denoted by $c^\ast_A(f )$, is the smallest integer $n$ such that $A=U_0 \cup \cdots \cup   U_n$,
where the $U_i$ are open in $A$ and each restriction ${f^{k_i}}_{\vert U_i}\colon  U_i \to X$ is null-homotopic for
some $k_i$.
\end{definition}

As always, we have normalized it.
We found the following interpretation in terms of the homotopic distance.

\begin{proposition}If $X$ is path connected. and $x_0\in X$, then the weak category of $f$ reduced to $A$ equals
$$c^\ast_A(f)= \lim_{k \to \infty}\Di_X(A;f^k,x_0).$$
\end{proposition}

\begin{proof}
First, observe that 
$$\Di_X(A;f^{k+1},x_0)=\Di_X(A;f^{k+1},f(x_0))\leq \Di_X(A;f^k,x_0),$$ 
because
$f^k \simeq x_0$ on $u_i$ implies $f^{k+1}\simeq f(x_0)$ on $U_i$, and the constant maps $x_0$ and $f(x_0)$ are homotopic.

If $c_A^\ast(f)\leq n$, we take an open covering $U_0\cup\cdots\cup U_n=A$ such that 
${f^{k_i}}_{\vert U_i}\simeq x_0$ for some $k_i$. By taking $k$ as the maximum of the $k_i$, we can assume that $k$ is the same for all $U_i$. Then $\Di_X(A;f^k,x_0)\leq n$, hence $\lim\nolimits_{k\to \infty} \Di(f^k,x_0)\leq n$.

On the other hand, if $\lim\nolimits_{k\to\infty} \Di_X(A;f^k,x_0)\leq n$, then
%by monoticity, 
$\Di_X(A;f^k,x_0)\leq n$ for some integer $k$. By definition of subspace distance, there is an open covering $U_0\cup\cdots\cup U_n=A$ such that
${f^k}_{\vert U_i} \simeq x_0$ for all $i$. Then $c_A^\ast(f)\leq n$. \end{proof}

\section{Navigation functions}\label{NAVIG}
Navigation functions  
exploit the gradient flow of a Morse-Bott function for constructing motion
planning algorithms.  

Originally, Koditschek and Rimon  (\cite{KR1990}) studied machines that navigate to a fixed goal using a gradient flow technique. Later, Farber (\cite{Fa2008}) considered navigation functions, as  in Theorem \ref{TC-CRIT},
which depend on two variables, the source and the target. He gave an explicit description of motion planning algortihms based in a navigation function. We shall adapt his explanation to our generalized setting, described in Section \ref{MOTPLAN}, as a direct application of Theorem \ref{MAINCUT}. 

Assume that we have two continuous maps $f,g\colon M \to Y$, defined on the manifold $M$, and that a Morse-Bott function $\Phi\colon M \to \R$ is given, with critical values $c_1,\dots,c_p$. 
\begin{itemize}
    \item
Call $r_i$  the subspace distance $\Di_M(\Sigma_i;f,g)$, for each critical level $\Sigma_i$ of $\Phi$. Find a decomposition $G_1^i\cup \cdots G_{r_i}^i=\Sigma_i$ of $\Sigma_i$ into ENRs, which solves the generalized motion planning in $\Sigma_i$ for the restrictions of $f$ and $g$. 
\item
Consider the basins of attraction $V_j^i\subset M$ of each $G_j^i$, $i=1,\dots,p$, $j=1,\dots,r_i$. If $x\in V_j^i$, we can move $x$ following the trajectory $x(t)$ of the negative gradient flow (may be in an infinite time), arriving to some point $\alpha\in G_j^i\subset \Sigma_i$. The path $f(x(t))$ connects $f(x)$ to $f(\alpha)$. We then consider the path $\gamma$
on $Y$ from $f(\alpha)$ to $g(\alpha)$ which solves the motion problem for $\alpha$ in $G_j^i$.
\item
Finally, we consider the image $g(\bar x(t))$ by $g$ of the inverse path $\bar x(t)$.
\end{itemize}

Since $x'(t)=(\grad \Phi)_{x(t)}$ for all $t$, we can reparametrize the flow by considering the change of variable
$$s=\int_0^t \vert (\grad \Phi)_{x(t)}\vert dt.$$
In this way, the trajectory $x(s)$ reaches the critical submanifold $\Sigma_i$ in a finite time.

The same ideas can be applied to the cut locus of a submanifold, in order to apply Theorem \ref{MAINCUT}. 

See \cite[Section 3.2]{GMP2011}, \cite[Section 5]{MPT2020} or \cite[Remark 2.12, Lemma 3.18 and Theorem 3.28]{BP2020} for   explicit computations of gradient flows.

\section*{acknowledgements}
We are grateful to Daniel Tanr\'e for pointing to us reference \cite{BP2020}.



\begin{thebibliography}{10}
%\providecommand{\url}[1]{{#1}}
%\providecommand{\urlprefix}{URL }
%\expandafter\ifx\csname urlstyle\endcsname\relax
  \providecommand{\doi}[1]{DOI~\discretionary{}{}{}#1}
  %\else
  %\providecommand{\doi}{DOI~\discretionary{}{}{}\begingroup
 % \urlstyle{rm}\Url}\fi

\bibitem{An2014}
Angulo-Ardoy, P.: Cut and conjugate points of the exponential map, with
  applications (2014).
\newblock Preprint arXiv:1411.3933

\bibitem{BH2004}
Banyaga, A., Hurtubise, D.E.: A proof of the {M}orse-{B}ott lemma.
\newblock Expo. Math. \textbf{22}(4), 365--373 (2004).
\newblock \doi{10.1016/S0723-0869(04)80014-8}.


\bibitem{BP2020}
Basu Somnath ;~Prasad, S.: A connection between cut locus, thom space and
  morse-bott functions (2020).
\newblock Preprint arXiv:2011.02972v1

\bibitem{BC2018}
{B{\l}aszczyk}, Z., {Carrasquel}, J.: {Topological complexity and efficiency of
  motion planning algorithms}.
\newblock {Rev. Mat. Iberoam.} \textbf{34}(4), 1679--1684 (2018)

\bibitem{Br1993}
Bredon, G.E.: Topology and geometry, \emph{Graduate Texts in Mathematics}, vol.
  139.
\newblock Springer-Verlag, New York (1993).
\newblock \doi{10.1007/978-1-4757-6848-0}.


\bibitem{Bu1977}
{Buchner}, M.A.: {Simplicial structure of the real analytic cut locus}.
\newblock {Proc. Am. Math. Soc.} \textbf{64}, 118--121 (1977)

\bibitem{CLOT2003}
Cornea, O., Lupton, G., Oprea, J., Tanr\'{e}, D.: Lusternik-{S}chnirelmann
  category, \emph{Mathematical Surveys and Monographs}, vol. 103.
\newblock American Mathematical Society, Providence, RI (2003).
\newblock \doi{10.1090/surv/103}.


\bibitem{Co2010}
Costa, A.E.: Topological complexity of configuration spaces.
\newblock Ph.D. thesis, Durham E-Theses (2010)

\bibitem{Cr1962}
{Crittenden}, R.: {Minimum and conjugate points in symmetric spaces}.
\newblock {Can. J. Math.} \textbf{14}, 320--328 (1962)

\bibitem{Di2008}
tom Dieck, T.: Algebraic topology.
\newblock EMS Textbooks in Mathematics. European Mathematical Society (EMS),
  Z\"{u}rich (2008).
\newblock \doi{10.4171/048}.


\bibitem{Do1972}
Dold, A.: Lectures on algebraic topology.
\newblock Springer-Verlag, New York-Berlin (1972).
\newblock Die Grundlehren der mathematischen Wissenschaften, Band 200

\bibitem{Fa2003}
Farber, M.: Topological complexity of motion planning.
\newblock Discrete Comput. Geom. \textbf{29}(2), 211--221 (2003).
\newblock \doi{10.1007/s00454-002-0760-9}.


\bibitem{Fa2008}
Farber, M.: Invitation to topological robotics.
\newblock Zurich Lectures in Advanced Mathematics. European Mathematical
  Society (EMS), Z\"{u}rich (2008).
\newblock \doi{10.4171/054}.


\bibitem{GHL1990}
Gallot, S., Hulin, D., Lafontaine, J.: Riemannian geometry, third edn.
\newblock Universitext. Springer-Verlag, Berlin (2004).
\newblock \doi{10.1007/978-3-642-18855-8}.


\bibitem{GS1978}
{Gluck}, H., {Singer}, D.: {Scattering of geodesic fields, I}.
\newblock {Ann. Math. (2)} \textbf{108}, 347--372 (1978)

\bibitem{GMP2011}
G\'{o}mez-Tato, A., Mac\'{\i}as-Virg\'{o}s, E., Pereira-S\'{a}ez, M.J.: Trace
  map, {C}ayley transform and {LS} category of {L}ie groups.
\newblock Ann. Global Anal. Geom. \textbf{39}(3), 325--335 (2011).
\newblock \doi{10.1007/s10455-010-9239-8}.


\bibitem{Ha2002}
{Hatcher}, A.: {Algebraic topology}.
\newblock Cambridge: Cambridge University Press (2002)

\bibitem{IT2001}
{Itoh}, J.-I., {Tanaka}, M.: {The Lipschitz continuity of the distance
  function to the cut locus}.
\newblock {Trans. Am. Math. Soc.} \textbf{353}(1), 21--40 (2001)

\bibitem{KM2011}
Kadzisa, H., Mimura, M.: Morse-{B}ott functions and the
  {L}usternik-{S}chnirelmann category.
\newblock J. Fixed Point Theory Appl. \textbf{10}(1), 63--85 (2011).
\newblock \doi{10.1007/s11784-010-0041-9}.


\bibitem{Ko1989}
{Kobayashi}, S.: {On conjugate and cut loci}.
\newblock {Global differential geometry, MAA Stud. Math. 27, 140-169} (1989)

\bibitem{KN1967}
Kobayashi, S., Nomizu, K.: Foundations of differential geometry. {V}ol.
  {I}-{II}.
\newblock Wiley Classics Library. John Wiley \& Sons, Inc., New York (1996).
\newblock Reprint of the 1963-1969 original, A Wiley-Interscience Publication

\bibitem{KR1990}
Koditschek, D.E., Rimon, E.: Robot navigation functions on manifolds with
  boundary.
\newblock Adv. in Appl. Math. \textbf{11}(4), 412--442 (1990).
\newblock \doi{10.1016/0196-8858(90)90017-S}.


\bibitem{LS1934}
{Lusternik}, L., {Schnirelmann}, L.: {M\'ethodes topologiques dans les
  probl\`emes variationnels. I. Pt. Espaces \`a un nombre fini de dimensions.
  Traduit du russe par J. Kravtchenko}.
\newblock {Actualit\'es Scientifiques et Industrielles. 188. Expos\'es sur
  l'analyse math\'ematique et ses applications III. Publi\'es par J. Hadamard.
  Paris: Hermann \& Cie. 51 S., 5} (1934)

\bibitem{MP2013}
Mac\'{\i}as-Virg\'{o}s, E., Pereira-S\'{a}ez, M.J.: An upper bound for the
  {L}usternik-{S}chnirelmann category of the symplectic group.
\newblock Math. Proc. Cambridge Philos. Soc. \textbf{155}(2), 271--276 (2013).
\newblock \doi{10.1017/S0305004113000200}.


\bibitem{MPT2017}
Mac\'{\i}as-Virg\'{o}s, E., Pereira-S\'{a}ez, M.J., Tanr\'{e}, D.: Morse theory
  and the {L}usternik-{S}chnirelmann category of quaternionic {G}rassmannians.
\newblock Proc. Edinb. Math. Soc. (2) \textbf{60}(2), 441--449 (2017).
\newblock \doi{10.1017/S0013091516000195}.


\bibitem{MPT2020}
{Mac\'{\i}as-Virg\'os}, E., {Pereira-S\'aez}, M.J., {Tanr\'e}, D.: {Non-linear
  Morse-Bott functions on quaternionic Stiefel manifolds}.
\newblock {Indag. Math., New Ser.} \textbf{31}(6), 968--983 (2020)

\bibitem{MM2020}
Mac\'ias-Virg\'os, E., Mosquera-Lois, D.: Homotopic distance between maps.
\newblock Mathematical Proceedings of the Cambridge Philosophical Society p.
  1--21 (2021).
\newblock \doi{10.1017/S0305004121000116}

\bibitem{MW2019}
Murillo, A., Wu, J.: Topological complexity of the work map.
\newblock Journal of Topology and Analysis \textbf{13}(01), 219--238 (2021).
\newblock \doi{10.1142/S179352532050003X}

\bibitem{Ni2007}
Nicolaescu, L.I.: An invitation to {M}orse theory.
\newblock Universitext. Springer, New York (2007)

\bibitem{Pa2017}
Pave\v{s}i\'{c}, P.: Complexity of the forward kinematic map.
\newblock Mech. Mach. Theory \textbf{A}, 230--243 (2017).


\bibitem{Pa2019}
Pave\v{s}i\'{c}, P.: Topological complexity of a map.
\newblock Homology Homotopy Appl. \textbf{21}(2), 107--130 (2019).
\newblock \doi{10.4310/HHA.2019.v21.n2.a7}.


\bibitem{Re1972}
Reeken, M.: Stability of critical points under small perturbations. {I}.
  {T}opological theory.
\newblock Manuscripta Math. \textbf{7}, 387--411 (1972).
\newblock \doi{10.1007/BF01644075}.


\bibitem{RS2003}
Rudyak, Y.B., Schlenk, F.: Lusternik-{S}chnirelmann theory for fixed points of
  maps.
\newblock Topol. Methods Nonlinear Anal. \textbf{21}(1), 171--194 (2003).
\newblock \doi{10.12775/TMNA.2003.011}.


\bibitem{Sakai}
Sakai, T.: Riemannian geometry.
\newblock Translations of Mathematical Monographs, 149. American Mathematical
  Society (1996)

\bibitem{Sc2020}
Scott, J.: On the topological complexity of maps (2020).
\newblock Preprint arXiv:2011.10646

\bibitem{Y2018}
{Yokoi}, K.: {Lusternik-Schnirelmann category based on the discrete Conley
  index theory}.
\newblock {Glasg. Math. J.} \textbf{61}(3), 693--704 (2019)




\end{thebibliography}
\end{document}